\documentclass[10pt]{amsart}

\usepackage[utf8]{inputenc}
\usepackage[T1]{fontenc}
\usepackage[english]{babel}
\usepackage{amsmath,amssymb,amsfonts,amsthm}
\usepackage{amsmath,amsthm,amsfonts,amssymb, hyperref}
\usepackage[pdftex]{graphicx}
\usepackage{enumitem}

\newtheorem{theorem}{Theorem}[section]

\newtheorem{proposition}[theorem]{Proposition}

\newtheorem{remark}[theorem]{Remark}

\newtheorem{example}[theorem]{Example}

\theoremstyle{definition}

\begin{document}
\title[Bounds between Laplace and Steklov eigenvalues]{Bounds between Laplace and Steklov eigenvalues on nonnegatively curved manifolds}
\begin{abstract}
Consider a compact Riemannian manifold with boundary. In this short note we prove that under certain positive curvature assumptions on the manifold and its boundary the Steklov eigenvalues of the manifold are controlled by the Laplace eigenvalues of the boundary. Additionally, in two dimensions we obtain an upper bound for these Steklov eigenvalues in terms of topology of the surface without any curvature restrictions.
\end{abstract}
\author[Mikhail Karpukhin]{Mikhail Karpukhin}
\address{Department of Mathematics and Statistics, McGill University, Burnside Hall,
805 Sherbrooke Street West,
Montreal, Quebec,
Canada,
H3A 0B9}
\email{mikhail.karpukhin@mail.mcgill.ca}
\maketitle

\section{Introduction and main results}

Let $(M,g)$ be a smooth compact Riemannian manifold with smooth boundary. We consider two elliptic self-adjoint operators defined on $\partial M$. The first is the usual Laplace-Beltrami operator $\Delta = \Delta_{\partial M}$ acting on $C^\infty(\partial M)$ with respect to the induced metric $g|_{\partial M}$. Since $\partial M$ is compact, the spectrum of $\Delta$ is discrete and consists of eigenvalues which we denote by
$$
0=\lambda_1\leqslant\lambda_2\leqslant\lambda_3\leqslant\ldots,
$$
where eigenvalues are counted with multiplicities. Note that $\partial M$ is not necessarily connected, so eigenvalue $0$ might not be simple. That is the reason why we start our numeration with $\lambda_1$ as opposed to $\lambda_0$.

The second operator is Dirichlet-to-Neumann operator $\mathcal{D}$, which is defined in the following way. For any $u\in C^\infty(\partial M)$, there is a unique harmonic extension $\hat u\in C^\infty(M)$, i.e. there is a unique $\hat u$ such that $\Delta_M\hat u = 0$ and $\hat u|_{\partial M} = u$. Then one defines $\mathcal{D}(u) = \partial_n\hat u$, where $n$ is an outward unit normal vector to $\partial M$. Similarly to the Laplacian the operator $\mathcal{D}$ is elliptic self-adjoint with discrete spectrum. We denote the corresponding sequence of eigenvalues by
$$
0=\sigma_1\leqslant\sigma_2\leqslant\sigma_3\leqslant\ldots,
$$
where eigenvalues are counted with multiplicities. Note again that the numeration starts with $\sigma_1$ as opposed to $\sigma_0$. The numbers $\sigma_i$ and the sequence $\{\sigma_i\}$ are sometimes refered to as Steklov eigenvalues of $M$ and Steklov spectrum of $M$ respectively. The study of Steklov spectrum has become rather popular in the recent years, see e.g. survey paper~\cite{GP} and references therein.

The aim of this paper is to show that under certain curvature restrictions on $M$ and $\partial M$ the sequence $\{\sigma_i\}$ is majorated by the sequence $\{\lambda_i\}$. The precise statement is the following.
\begin{theorem}
\label{MainTheorem}
Let $(M,g)$ be a compact Riemannian manifold of dimension $n\geqslant 3$ with boundary. Suppose that the Weitzenb\"ock curvature on $2$-forms $W^{[2]}$ is nonnegative and the lowest $(n-2)$-curvature $c_{n-2}$ of $\partial M$ is bounded from below by a positive constant $c$. Then, if $n\geqslant 4$, for all $m\geqslant 1$
one has the inequality
\begin{equation}
\label{MainInequality}
\sigma_m\leqslant \frac{n-2}{(n-1)c}\lambda_m.
\end{equation}
Moreover, if $n=3$ then for all $m > b_0(M)$
$$
\sigma_m < \frac{2}{3c}\lambda_m,
$$ 
where $b_0(M)$ is the number of connected components of $M$.
\end{theorem}
In section~\ref{background} we define the Weitzenb\"ock curvature on $p$-forms and the lowest $p$-curvatures of the boundary for all $p=1,\ldots,n$. For now, let us mention that convexity of $\partial M$ is enough for boundedness of $c_p$ for all $p$, whereas nonnegativity of the curvature operator suffices for nonnegativity of $W^{[p]}$ for all $p$.

\begin{remark}
Bound~(\ref{MainInequality}) is sharp for $m=2$ on a Euclidean ball $\mathbb{B}^n(\frac{1}{c})$ of radius $\frac{1}{c}$. Indeed, on the standard sphere $\mathbb{S}^{n-1}(\frac{1}{c})$ of radius $\frac{1}{c}$ all principal curvatures are constant and equal to $c$, therefore, $c_{n-2} = (n-2)c$ (see definition in Section~\ref{background}). Additionally, it is well-known that $\sigma_2(\mathbb{S}^{n-1}(\frac{1}{c})) = c$, while $\lambda_2(\mathbb{S}(\frac{1}{c})) = (n-1)c^2$. Putting all these numbers together, one obtains an equality in~(\ref{MainInequality}). Moreover, both those eigenvalues have multiplicity $n$, therefore, we have that inequality~(\ref{MainInequality}) is sharp for all $m\leqslant n+1$. However, as it follows from the remark after Theorem~\ref{ThRS}, the equality in inequality~(\ref{MainInequality}) can only be achieved on a ball, i.~e. for $m>n+1$ inequality~(\ref{MainInequality}) is strict.
\end{remark}
\begin{remark}
For a fixed manifold $M$ bound of the type~(\ref{MainInequality}), i.e. inequality bounding $\sigma_m$ in terms of a linear function of $\lambda_m$, can not possibly be sharp for all $k$. Indeed, according to Weyl's law (see e.g.~\cite{GP}) $\sigma_m\sim m^{\frac{1}{n-1}}$, whereas $\lambda_m\sim m^{\frac{2}{n-1}}$. 
\end{remark}

A curious feature of the proof is that it uses the extension of the Dirichlet-to-Neumann map to the space of differential forms of higher degree defined by Raulot and Savo~\cite{RS}, whereas the final statement does not refer to differential forms. 

In case $M$ is a surface we have the following theorem.
\begin{theorem}
\label{Th2}
Let $(M,g)$ be a connected oriented Riemannian surface. Then one has
\begin{equation}
\label{steklov2}
\sigma_{p+1}\sigma_{q+1}L^2(\partial M)\leqslant
\begin{cases}
\pi^2(p+q+2\gamma+2k-3)^2 &\text{if  $p+q \equiv 1$} \\
\pi^2(p+q+2\gamma+2k-2)^2 &\text{if  $p+q \equiv 0$}
\end{cases}
\pmod{2},
\end{equation}
where $L(\partial M)$ is the length of $\partial M$, $k$ is the number of boundary components and $\gamma$ is the genus of $M$.
In particular, when $p=q$ one obtains,
\begin{equation}
\label{steklov}
\sigma_{p+1}L(\partial M) \leqslant 2\pi(p+\gamma+k-1).
\end{equation}
\end{theorem}

\begin{remark} Theorem~\ref{Th2} is a consequence of Theorem~\ref{ThYY} below. As it is explained in~\cite{YY} Theorem~\ref{ThYY} can be seen as a generalization of Hersh-Payne-Schiffer inequality~\cite{HPS}, which is inequality~(\ref{steklov2}) for simply connected planar domains, i.~e. $k=1$, $\gamma = 0$. Theorem~\ref{Th2} makes the connection more evident. A particular case $k=1$ of inequality~(\ref{steklov2}) have already been pointed out in~\cite{YY}. A similar inequality was obtained by Girouard and Polterovich in~\cite{GP2}
$$
\sigma_{p+1}\sigma_{q+1}L^2(\partial M)\leqslant
\begin{cases}
\pi^2(p+q-1)^2(\gamma+ k)^2 &\text{if  $p+q \equiv 1$} \\
\pi^2(p+q)^2(\gamma+k)^2 &\text{if  $p+q \equiv 0$}
\end{cases}
\pmod{2}.
$$
However, one can easily see that inequality~(\ref{steklov2}) yields a better upper bound unless either $\gamma = 0$, $k=1$ or $(p,q) = (1,1), (2,1)$, when both inequalities yield the same bound. Let us also mention a remarkable series of papers by Fraser and Schoen~\cite{FS1, FS2}, where the authors investigate the optimal upper bound of the form~(\ref{steklov}) for $p=1$.
Note that in the mentioned papers the enumeration of Steklov eigenvalues starts with $\sigma_0$, which is compensated by the fact that we have $p+1$ and $q+1$ in the left hand side of~(\ref{steklov2}).
\end{remark}
\begin{remark}
Bound~(\ref{steklov}) has an advantage of being linear in all the parameters involved. In fact, Hassanezhad~\cite{Asma} proved that there exists a bound of the form
$$
\sigma_{p+1}L(\partial M)\leqslant A\gamma + B p
$$
with implicit universal constants $A$ and $B$. Thus, by introducing linear dependence on $k$ we are able to make the constants in the above bound explicit.
Let us also mention that for $p=1$, Kokarev proved in~\cite{Kokarev} an upper bound
$$
\sigma_2L(\partial M)\leqslant 8\pi\left[\frac{\gamma+3}{2}\right].
$$
\end{remark}

The paper is organized as follows. Section~\ref{background} contains necessary background from differential geometry. In Section~\ref{DNoperator} we describe the extension of the operator $\mathcal{D}$ to differential forms due to Raulot and Savo. Theorem~\ref{Th2} is proved in Section~\ref{digresion}. We prove Theorem~\ref{MainTheorem} for orientable manifolds in Section~\ref{proof} and in the subsequent section we show how to generalise this proof to nonorientable manifolds. Finally in the last section we compare our results to a similar result due to Wang and Xia~\cite{WX}.  

\section{Background in differential geometry}

\label{background}
Let $\nabla$ denote the Levi-Civita connection on $M$ associated to the metric $g$. The operator $\nabla$ and the metric $g$ have natural extensions to the bundle $\bigwedge^p(T^*M)$ of differential $p$-forms on $M$ for all $p=1,\ldots,n$, which will be denoted by the same letters. Let $\nabla^*$ be the formal adjoint of $\nabla$. Then the Weitzenb\"ock curvature $W^{[p]}$ on $p$-forms is defined by the Bochner formula
\begin{equation}
\label{Bochner}
\Delta_M\omega = \nabla^*\nabla\omega + W^{[p]} \omega,
\end{equation}
where $\omega\in\Omega^p(M) = \Gamma(\bigwedge^p(T^*M))$ and $\Delta$ is the Hodge Laplacian on $\Omega^p(M)$.

According to~\cite{Peterson} the condition $W^{[2]}\geqslant 0$ can be expressed in terms of Riemann curvature tensor $R(\cdot,\cdot,\cdot,\cdot)$. In fact, non-negativity of $W^{[2]}$ is equivalent to requiring non-negativity of the following expressions:
\begin{itemize}
\item {\em Second Ricci curvature:}
\begin{equation}
\label{2ndRC}
R(x,u,u,x) + R(y,u,u,y) \geqslant 0
\end{equation}
for any orthonormal triplet $x,u,y$.
\item {\em Isotropic curvature:}
$$
R(x,u,u,x) + R(y,v,v,y) + R(x,v,v,x) + R(y,u,u,y) + 2R(x,y,u,v) \geqslant 0
$$
for any orthonormal quadruplet $x,u,y,v$.
\end{itemize}
By summing inequality~(\ref{2ndRC}) over $x,y$ from some orthonormal basis of $u^\perp$ in the tangent space of $M$, one obtains $Ric(u,u)\geqslant 0$, i.~e. $W^{[2]}\geqslant 0$ implies $Ric \geqslant 0$. Note that the non-negativity of the curvature operator implies non-negativity of $W^{[p]}$ for all $p$, see e.g. the same paper~\cite{Peterson}.

The principal curvatures $\eta_1,\ldots,\eta_{n-1}$ of $\partial M\subset M$ are defined as eigenvalues of the second fundamental form of $\partial M$. Then the lowest $p$-curvature $c_p$ of $\partial M$ is defined as $\min_{i_1,\ldots,i_p} (\eta_{i_1} + \ldots +\eta_{i_p})$, where $\{i_1,\ldots,i_p\}$ range over all $p$-element subsets of $\{1,\ldots, n-1\}$. This way $c_1>0$ is equivalent to the convexity of $\partial M$ and $c_{n-1}$ is proportional to the mean curvature of $\partial M$ in $M$.

\section{Dirichlet-to-Neumann operator on differential forms}

\label{DNoperator}
There are several definitions of Dirichlet-to-Neumann operator on forms, see e.g.~\cite{BS, JL, RS, SS}. Here we discuss the one due to Raulot and Savo. For $\omega\in\Omega^p(\partial M)$ there exists a unique differential form $\hat\omega\in\Omega^p(M)$, see e.g.~\cite{Sch}, such that 
$$
\Delta\hat\omega = 0,\quad \hat\omega |_{\partial M} = \omega, \quad i_n\hat\omega = 0,
$$
where $i_n\hat\omega$ stands for contraction of $\hat\omega$ with the unit outer normal vector field $n$, i.e. $\hat\omega(\cdot,\cdot,\ldots,\cdot) = \hat\omega(n,\cdot,\ldots,\cdot)$, which (since $n$ is only defined on the boundary) yields a well-defined element of $\Omega^{p-1}(\partial M)$.
Then the Dirichlet-to-Neumann operator $\mathcal{D}^{(p)}$ on $\Omega^p(\partial M)$ is defined as $\mathcal{D}^{(p)}\omega = i_nd\hat\omega$. Raulot and Savo proved that $\mathcal{D}^{(p)}$ is a positive elliptic self-adjoint operator and therefore its spectrum consists of a sequence of eigenvalues
$$
0\leqslant\sigma_1^{(p)}\leqslant\sigma_2^{(p)}\leqslant\sigma_3^{(p)}\leqslant\ldots,
$$
where eigenvalues are written with multiplicities. Since $\mathcal{D}^{(0)} = \mathcal{D}$, we sometimes interchange notations $\sigma_i^{(0)}$ and $\sigma_i$ for the $i$-th eigenvalue of the Dirichlet-to-Neumann operator on $C^\infty(\partial M)$. As in the case of $\mathcal{D}$, the numbers $\sigma_i^{(p)}$ are sometimes referred to as Steklov eigenvalues of $M$ on differential $p$-forms.

This particular definition of $\mathcal{D}^{(p)}$ is of interest to us due to the following two theorems.

\begin{theorem}[Yang, Yu~\cite{YY}] 
\label{ThYY}
Let $(M,g)$ be a compact oriented $n$-dimensional Riemannian manifold with nonempty boundary. Let $\sigma^{(p)}_m$ be the $m$-th Steklov eigenvalue on differential $p$-forms of $M$ and $\lambda_m$ be the $m$-th eigenvalue for the Laplacian operator on $\partial M$. Then for any two positive integers $m$ and $r$, one has 
\begin{equation}
\label{ThYYeq}
\sigma_m^{(0)}\sigma^{(n-2)}_{b_{n-2}+r}\leqslant \lambda_{m+r+b_{n-1}-1},
\end{equation}
where $b_k$ is the $k$-th Betti number of $M$.
\end{theorem}

\begin{theorem}[Raulot, Savo~\cite{RS}]
\label{ThRS}
Let $(M,g)$ be a compact $n$-dimensional Riemannian manifold with nonempty boundary. Let $p=1,\ldots,n-1$. Assume that $M$ satisfies $W^{[p]}\geqslant 0$ and that $c_p\geqslant c > 0$. Then if $p<\frac{n}{2}$ one has the inequality
\begin{equation}
\label{ThRSeq1}
\sigma_1^{(p)} > \frac{n-p+1}{n-p}c.
\end{equation}
If $p\geqslant \frac{n}{2}$ one has
\begin{equation}
\label{ThRSeq2}
\sigma_1^{(p)} \geqslant \frac{p+1}{p}c.
\end{equation}
\end{theorem}
\begin{remark}
As it was pointed out in~\cite{RS}, $\dim\ker \mathcal{D}^{(p)} = b_{p}$. Since $\sigma_1^{(p)}>0$ iff $\dim\ker \mathcal{D}^{(p)} =0$, it implies that under assumptions of the theorem one has $b_p = 0$.
\end{remark}
\begin{remark}
In the same paper~\cite{RS} it is also proved that equality in~(\ref{ThRSeq2}) implies that $M$ is a Euclidean ball of radius $\frac{1}{c}$.
\end{remark}

\section{Proof of Theorem~\ref{Th2}}

\label{digresion}
For a connected oriented surface $n=2$ inequality~(\ref{ThYYeq}) with $m = p$ and $r = q$ reads
\begin{equation}
\label{corollary}
\sigma_{p}\sigma_{q+1}\leqslant \lambda_{p+q+(k+2\gamma-1)-1},
\end{equation}
where $k$ is the number of boundary components of $M$ and $\gamma$ is the genus of $M$. In this section we would like to obtain an upper bound for $\sigma_p$ independent of the spectrum of the boundary. For simplicity of exposition let us only prove inequality~(\ref{steklov}). As one could see, the proof presented below can be easily modified to cover inequality~(\ref{steklov2}). Assume that $p= q+1$, then inequality~(\ref{corollary}) takes the following form
\begin{equation}
\label{corollary2}
\sigma^2_{p+1}\leqslant \lambda_{2(p+\gamma)+k-1}
\end{equation}

Let us introduce some notation. For each vector $l = (l_1,\ldots,l_k)\in \left(\mathbb{R}_{>0}\right)^k$ we introduce the sequence $S_l$ in the following way.
For each $i=1,\ldots,k$ consider an arithmetic progression $T_i=\left\{ \frac{d}{l_i}\right\}^\infty_{d=1}$. Then $S_l$ is a nondecreasing sequence obtained by taking the union of sequences $T_i$ and reordering the entries. The sequence $S_l$ in the context of Steklov problem first appeared in the paper~\cite{GPPS}. It is easy to see that the right hand side of~(\ref{corollary2}) is $4\pi^2S_l[p+\gamma]^2$, where $l_i$ is the length of the $i$-th boundary component of $M$ and $A[n]$ denotes the $n$-th entry of a sequence $A$. Denoting by $|l|$ the $L^1$ norm of $l$, i.~e. $|l| = l_1+\ldots + l_k$, inequality~(\ref{corollary2}) implies
\begin{equation}
\label{corollary3}
\sigma^2_{p+1}|l|^2\leqslant 4\pi^2\sup\limits_{l} (|l|S_l[p+\gamma])^2.
\end{equation}
Thus, in order to obtain a universal upper bound on $\sigma_{p+1}$ it is enough to evaluate the right hand side of inequality~(\ref{corollary3}).
\begin{proposition}
\label{algebra}
For any $d\in\mathbb{N}$ one has 
$$
\sup\limits_{l} |l|S_{l}[d] = d+k-1.
$$
\end{proposition}
\begin{proof}
Let us first assume that there exists a vector $l=(l_1,\ldots, l_k)$ maximizing the fucntion $|l|S_{l}[d]$ and let $M$ denote the corresponding maximal value. 

{\bf Claim 1.} For each $i=1,\ldots,k$ there is $d_i\in\mathbb{N}$ such that $\frac{d_i}{l_i} = \frac{M}{|l|}$. Indeed, let $d_i$ be the smallest integer, such that $\frac{d_i}{l_i}\geqslant \frac{M}{|l|}$. Assume that $\frac{d_i}{l_i}$ is strictly greater than $\frac{M}{|l|}$. We will show that in that case $M$ is not a maximum. Indeed, by definition of $d_i$ it is easy to see that  
$$
S_{(l_1,\ldots,l_i,\ldots,l_k)}[d] = S_{(l_1,\ldots,\frac{d_i|l|}{M},\ldots,l_k)}[d] = \frac{M}{|l|}.
$$
At the same time, $\frac{d_i|l|}{M}>l_i$, therefore the replacement $l_i\mapsto \frac{d_i|l|}{M}$ increases the value of the function $|l|S_l[d]$. We arrive at a contradiction. 

{\bf Claim 2.} $S_{(l_1,\ldots,l_k)}[d] > S_{(l_1,\ldots,l_k)}[d-1]$. The proof of this claim is similar to the previous one. Indeed, by previous claim $l_1=\frac{d_1|l|}{M}$. Then, arguing by contradiction, one has $S_{l}[d] = S_l[d-1]= S_{(l'_1,\ldots,l_k)}[d]$, where $l'_1 = \frac{(d_1+1)|l|}{M}$. Once again $l'_1>l_1$ and we arrive at a contradiction. The component $l_1$ was chosen only for the notation convenience, the same argument follows through for any $l_i$.

According to the first claim, $l_i = \frac{d_i|l|}{M}$. Hence, for each $i$ the sequence $T_i$ has exactly $d_i - 1$ entries less than $\frac{M}{l}$. Therefore, the sequence $S_l$ contains $\Sigma(d_i-1)$ entries less than $\frac{M}{|l|}$. Since $S_l[d] = \frac{M}{|l|}$, Claim 2 implies 
\begin{equation}
\label{sum}
d = \Sigma_{i=1}^k (d_i-1) + 1.
\end{equation}
At the same time, $\Sigma\, l_i = |l|$, which implies $\Sigma \,d_i = M$. Substituting this into equality~(\ref{sum}) one obtains $M = d+k-1$. Additionally, note that this argument provides an explicit formula for all maximizers. Namely, any maximizer $l$ should have the following form 
$$
l = \frac{|l|}{d+k-1}(d_1,\ldots,d_k),
$$
where $d_i\in\mathbb{N}$ are such that $\Sigma\, d_i = d+k-1$. It is easy to see that any such vector satisfies $|l|S_{l}[d] = d+k-1$. Regardless of the existence of maximizer the arguments above prove that
\begin{equation}
\label{maximizer}
\sup\limits_{l} |l|S_{l}[d] \geqslant d+k-1,
\end{equation}
where one has equality provided the existence of a maximizer.

Let us turn to proving that the supremum is achieved. First, note that for any positive constant $\lambda>0$, one has $S_{\lambda l} = \frac{1}{\lambda}S_{l}$, therefore the function $|l|S_{l}[d]$ is scale-invariant for any $d$. Therefore, without loss of generality we may assume that $|l| = l_1+\ldots+l_k = 1$. Although the function $S_l[d]$ is obviously continuous on $L_k$, the space $L_k $ of vectors $l$ of unit length is not compact as $l_i$ are required to be positive. To circumvent this difficulty one needs to control the behaviour of the functional $S_{l}[d]$ near the boundary $\partial L_k$. 
To that end it is useful to note that $S_l[d]$ is bounded as it follows from the inequality $S_l[d]\leqslant T_i[d] = \frac{d}{l_i}$ which holds for all $i$. Since $|l| = 1$ there exists an index $i$ with $l_i\geqslant \frac{1}{k}$ and, therefore,
\begin{equation}
\label{boundedness}
S_l[d]\leqslant kd.
\end{equation}

The space $L_k$ is an interior of a $(k-1)$-simplex whose boundary is a union of faces $L_k^{(i)} = \{l\,|\, l_i = 0\}$. Consider an open neighbourhood $U^{(i)}\subset L_k$ of $L_k^{(i)}$ defined by $U^{(i)} = \{l\,|\, 0<l_i<\frac{1}{kd}\}$. Then one has $T_i[1]>kd\geqslant S_l[d]$ by inequality~(\ref{boundedness}). Thus, the presence of that $l_i$ does not affect $S_l[d]$ and setting $l^{(i)} = (l_1,\ldots,\hat l_i,\ldots,l_k)\in \mathbb{R}^{k-1}$ (here we use the hat symbol to denote the omission) one has the following inequality for all $l\in U^{(i)}$,
\begin{equation}
\label{induction}
S_l[d] =S_{l^{(i)}}[d] = \frac{1}{1-l_i}(|l^{(i)}|S_{l^{(i)}}[d])< \frac{kd}{kd-1}(|l^{(i)}|S_{l^{(i)}}[d]).
\end{equation}

Now we collect the facts to complete the proof.
We proceed by induction on the number of boundary components $k$. The base $k = 1$ is obvious since in that case the space $L_1$ is a single point. The step of induction guarantees that for $l'\in\left(\mathbb{R}_{>0}\right)^{k-1}$ one has $\sup |l'|S_{l'}[d] = d+k-2$. Combining it with the right hand side of inequality~(\ref{induction}) one obtains that for $l\in U^{(i)}$ the following inequality holds,
$$
S_l[d]< \frac{kd}{kd-1}(k+d-2)\leqslant k+d-1\leqslant \sup\limits_{l\in\left(\mathbb{R}_{>0}\right)^k} |l|S_l[d] = \sup\limits_{l\in L_k} S_l[d],
$$
where we used inequality~(\ref{maximizer}) and obvious observations $d\geqslant 1$ and $k\geqslant 2$. Therefore, the supremum of $S_l[d]$ over $L_k$ coincides with the supremum of $S_l[d]$ over compact subset $L_k\backslash\cup_{i=1}^kU^{(i)}$, i.~e. supremum is achieved and is equal to $d+k-1$ by inequality~(\ref{maximizer}).
\end{proof}

Note that the above proposition is purely algebraic. Nevertheless, as a corollary we obtain Theorem~\ref{Th2}.

\section{Proof of Theorem~\ref{MainTheorem} in the orientable case}

\label{proof}
The proof is obtained by an easy combination of Theorem~\ref{ThYY} and Theorem~\ref{ThRS}. First of all, let us make a couple of observations.

{\bf Observation 1.} Applying Hodge $*$-operator to both sides of formula~(\ref{Bochner}) implies that $W^{[n-p]}* = *W^{[p]}$. Indeed, both $\Delta$ and $\nabla^*\nabla$ are formally self-adjoint, which implies that they commute with $*$, therefore so does $W$. Thus, nonnegativity of $W^{[p]}$ is equivalent to nonnegativity of $W^{[n-p]}$.

{\bf Observation 2.} Application of Theorem~\ref{ThRS} for $p=n-2$ now yields that under the conditions of Theorem~\ref{MainTheorem} one has the inequality
$$
\sigma_1^{(n-2)} \geqslant \frac{n-1}{n-2}c
$$
if $n\geqslant 4$ and 
$$
\sigma_1^{(1)} > \frac{3}{2}c
$$
if $n=3$. In particular, by the remark after Theorem~\ref{ThRS} it yields $b_{n-2} = 0$. Similarly, since $W^{[2]}\geqslant 0$ implies $Ric = W^{[1]}\geqslant 0$ and $c_{n-2}>0$ implies $c_{n-1}>0$, we can apply Theorem~\ref{ThRS} for $p=n-1$ and conclude that $b_{n-1} = 0$.

Taking into account $b_{n-2} = b_{n-1} = 0$, an application of Theorem~\ref{ThYY} for $q=1$ yields the inequality
$$
\sigma_m\sigma_1^{(n-2)}\leqslant\lambda_m.
$$
Combining this inequality with inequlities from Observation 2, one completes the proof of Theorem~\ref{MainTheorem} for orientable manifolds.

\section{Non-orientable manifolds}

It is possible to generalize Theorem~\ref{ThYY} to include the case of non-orientable manifolds in a way that allows us to apply the arguments of the previous section. Below we give the statement and the short outline of the argument.

Let us pass to an orientable double cover $\pi\colon\tilde M\to M$ and endow it with the metric $\pi^*g$ such that the involution $\tau$ exchanging the leaves of $\pi$ is an isometry. Then $\tau$ induces the decomposition of differential forms on $\tilde M$ into even and odd with respect to $\tau$ which is compatible with Hodge-Morrey decomposition (for details on Hodge-Morrey decomposition see~\cite{YY}). Similarly, the eigenvalues of Dirichlet-to-Neumann and Laplace operators are divided into those corresponding to odd and even eigenforms respectively. If we denote by $\lambda_{i,even}$ the $i$-th even eigenvalue of $\Delta$ on $C^\infty(\partial \tilde M)$ and similarly by $\sigma^{(p)}_{j,odd}$ the $j$-th odd Steklov eigenvalue on $\Omega^p(\partial \tilde M)$, then inequality~(\ref{ThYYeq}) for $\tilde M$ becomes
\begin{equation}
\label{nonorientable}
\sigma^{(0)}_{m,even}\sigma^{(n-2)}_{(b_{n-2}(\tilde M) - b_{n-2}(M))+r, odd}\leqslant \lambda_{m+r+b_{n-1}(M)-1, even}.
\end{equation}

In order to prove~(\ref{nonorientable}) one follows the proof of Theorem~\ref{ThYY} accounting for the presence of $\tau$. The key points are the following:
\begin{itemize}
\item differential $d$ is compatible with the decomposition into odd and even forms;
\item Hodge star operator sends odd forms to even and vice versa;
\item Since the kernel of $\mathcal{D}^{(p)}$ consists of Neumann harmonic $p$-fields and even Neumann harmonic fields on $\tilde M$ are the pullbacks of Neumann harmonic fields on $M$, then $\dim(\ker \mathcal{D}^{(p)}\cap \Omega^p_{even}) = b_p(M)$, $\dim(\ker \mathcal{D}^{(p)}\cap \Omega^p_{odd}) = b_p(\tilde M) - b_p(M)$.
\end{itemize}

To complete the proof of Theorem~\ref{MainTheorem} one makes the following two observations. First, local curvature conditions on $M$ pass to $\tilde M$, therefore Betti numbers disappear from the inequality~(\ref{nonorientable}). Second, there is an obvious inequality $\sigma^{(n-2)}_{1,odd}\geqslant\sigma^{(n-2)}_1$. Since even eigenvalues of $\tilde M$ coincide with the eigenvalues of $M$, the same arguments as in Section~\ref{proof} conclude the proof.

\section{Comparison with earlier results}

\label{comparison}
In this section we compare results of Theorem~\ref{MainTheorem} with the following statement.
\begin{theorem}[Wang, Xia~\cite{WX}]
\label{ThWX}
Let $(M,g)$ be an $n$-dimensional compact connected Riemannian manifold with non-negative Ricci curvature and boundary $\partial M$. Assume that the principal curvatures of $\partial M$ are bounded from below by a positive constant $c$. Then one has the inequality
\begin{equation}
\label{ThWXeq}
\sigma_2(M)\leqslant \frac{\sqrt{\lambda_2(\partial M)}}{(n-1)c}\left(\sqrt{\lambda_2(\partial M)} + \sqrt{\lambda_2(\partial M) - (n-1)c^2}\right).
\end{equation}
\end{theorem}

We start our comparison with the case $n=3$, where conditions of our theorem coincide with conditions of Theorem~\ref{ThWX} (since $W^{[1]} = *W^{[2]}* = Ric$). Theorem~\ref{MainTheorem} yields for $m=2$ and $n=3$
\begin{equation}
\label{1}
\sigma_2(M)<\frac{2}{3c}\lambda_2(\partial M),
\end{equation}
while inequality~(\ref{ThWXeq}) for $n=3$ becomes
\begin{equation}
\label{2}
\sigma_2(M)\leqslant \frac{\sqrt{\lambda_2(\partial M)}}{2c}\left(\sqrt{\lambda_2(\partial M)} + \sqrt{\lambda_1(\partial M) - 2c^2}\right).
\end{equation}

It is easy to see now that the inequality~(\ref{1}) yields a better bound once $\lambda_2(\partial M)\geqslant \frac{9}{4}c^2$. 
\begin{example}
Let $M = \left.\{(x,y,z)\in\mathbb{R}^3\right|\, \frac{x^2}{a^2} + \frac{y^2}{b^2} + \frac{z^2}{c^2} \leqslant 1\}$. Then $Ric(M) = 0$ and $M$ is convex, i.e. it satisfies the conditions of the theorem. If $a>b>c$ the lower bound for principal curvatures is $\frac{c}{a^2}$ while the lower bound $K_0$ for Gaussian curvature of $\partial M$ is $\frac{c^2}{a^2b^2}$. Using the classical bound $\lambda_2(\partial M)\geqslant 2 K_0$ (see e.~g. p. 186 of~\cite{Chavel}) one sees that the bound given by Theorem~\ref{MainTheorem} is better once $b^2\leqslant \frac{8}{9}a^2$, i.e. once the ellipsoid is oblong enough.
\end{example} 

In case $n\geqslant 4$ our condition $W^{[2]}\geqslant 0$ implies condition $Ric\geqslant 0$ of Theorem~\ref{ThWX}. At the same time, the assumption "principle curvatures are bounded from below by $c>0$" of Theorem~\ref{ThWX} implies the the lowest $(n-2)$-curvature is bounded from below by $(n-2)c$. Thus, while our condition on the interior curvature is stronger, the condition on the curvature of $\partial M$ is weaker. Now, suppose that $n\geqslant 4$ and $(M,g)$ satisfies the conditions of both Theorem~\ref{MainTheorem} and Theorem~\ref{ThWX}, i.e. $W^{[2]}\geqslant 0$ and all principle curvatures of $\partial M$ are bounded from below by $c>0$, then the lowest $(n-2)$-curvature is bounded from below by $(n-2)c$. An application of Theorem~\ref{ThWX} yields inequality~(\ref{ThWXeq}), which can be rewritten as
$$
\sigma_2\leqslant \frac{\lambda_2}{(n-1)c} + \frac{\sqrt{\lambda_2(\lambda_2-(n-1)c^2)}}{(n-1)c},
$$
while Theorem~\ref{MainTheorem} for $m=1$ gives 
$$
\sigma_2\leqslant \frac{\lambda_2}{(n-1)c},
$$
which is clearly stronger. The right hand sides of those inequalities differ by an expression, which is equal to zero iff $\lambda_2 = (n-1)c^2$. According to the results of Xia~\cite{Xia} the latter happens only if $M$ is a Euclidean ball of radius $\frac{1}{c}$. Therefore, our estimate yields a better bound for any convex bounded domain of the Euclidean space other than a ball. Moreover, our estimate can also be applied to nonconvex domains, since our condition on the curvature of $\partial M$ is only $c_{n-2}>0$.

\subsection*{Acknowledgements.} The author is grateful to Iosif Polterovich for fruitful discussions and comments on the initial versions of the manuscript. The author thanks Spiro Karigiannis for providing the reference~\cite{Peterson}.

This research was partially supported by Tomlinson Fellowship. This work is a part of the author's PhD thesis at McGill University under the supervision of Dmitry Jakobson and Iosif Polterovich.

\end{document}